\documentclass[a4paper,11pt]{article}
\usepackage[T2A]{fontenc}
\usepackage[cp1251]{inputenc}
\usepackage[ukrainian, english]{babel}
\usepackage{amssymb,amsfonts,amsthm,amsmath,mathtext,cite,enumerate,float}
\usepackage[onehalfspacing]{setspace}
\usepackage{geometry}
\usepackage{graphicx}
\usepackage{multirow}
\usepackage{indentfirst}
\usepackage{color}

\newcommand{\aut}[1][n]{B_{#1}}
\newcommand{\syl}[1][n]{G_{#1}}

\newcommand{\m}{^{-1}}

\newcommand{\rest}[1][n]{_{(#1)}}

\newtheorem{theorem}{Theorem}
\newtheorem{proposition}[theorem]{Proposition}
\newtheorem{corollary}[theorem]{Corollary}
\newtheorem{lemma}[theorem]{Lemma}

\newtheorem{remark}{Remark}

\theoremstyle{definition}
\newtheorem{definition}{Definition}

\begin{document}
%\title{Commutator of Sylow 2-subgroups of alternating groups and symmetric group and their description}
%\title{Commutator and description of Sylow 2-subgroups of alternating groups and symmetric group}
\title{The commutator and centralizer description of Sylow subgroups of alternating and symmetric groups}
\author{Ruslan Skuratovskii} %,
%\date{}
\maketitle

%\begin{center}
%\parbox{11.5cm}
%{\bf {Description and commutator
%} \/}

%\title[] {Description and commutator of Sylow 2-subgroups of alternating groups and symmetric group}
%\author[]{Ruslan Skuratovskii,  }

  \begin{abstract}
  %\note{Rewrite abstract.}

  Given a permutational wreath product sequence of cyclic groups of order we research a  commutator width of such groups and some properties of its commutator subgroup.
 Commutator width of Sylow 2-subgroups of alternating group ${A_{{2^{k}}}}$, permutation group ${S_{{2^{k}}}}$ and $C_p \wr B$ were founded. The result of research was extended on subgroups $(Syl_2 {A_{{2^{k}}}})'$, $p>2$.
The paper presents a construction of commutator subgroup of Sylow 2-subgroups of symmetric and alternating groups. Also minimal generic sets of Sylow 2-subgroups of $A_{2^k}$ were founded.
Elements presentation of $(Syl_2 {A_{{2^{k}}}})'$, $(Syl_2 {S_{{2^{k}}}})'$ was investigated.
We prove that the commutator width \cite {Mur} of an arbitrary element of a discrete wreath product of cyclic groups $C_{p_i}, \, p_i\in \mathbb{N} $ is 1.
% kpable decision

  \end{abstract}

% \begin{Title}  \end{Title}

  \begin{section}{Introduction}
  % \note{improve Introduction}

A form of commutators of wreath product $A\wr B$ was briefly considered in \cite{Meld}. For more deep description of this form we take into account the commutator width $(cw(G))$ which was presented in work of Muranov  \cite {Mur}. This form of commutators of wreath product was used by us for the research of  $cw(Syl_2 {A_{{2^{k}}}})$, $cw(Syl_2 {S_{{2^{k}}}})$ and $cw (C_p \wr B)$.  As well known, the first example of a group $G$ with $cw(G) > 1$ was given by
Fite \cite{Fite}. We deduce an estimation for commutator width of wreath product $B \wr C_p$ %$ C_p \wr B$
 of groups $C_p$ and an arbitrary group $B$ taking into the consideration a $cw(B)$ of passive group $B$. In this paper we continue a researches which was stared in \cite{SkAr}. The form of commutator presentation \cite{Meld} was presented by us in form of wreath recursion and commutator width of it was studied.
A research of commutator-group
 serves to decision  of inclusion problem \cite{Lin} for elements of $Syl_2 {A_{{2^{k}}}}$ in its derived subgroup $(Syl_2 {A_{{2^{k}}}})'$.
% kpable decision
  \end{section}

  \begin{section}{Preliminaries }
Denote by $fun(B,A)$ the direct product of isomorphic copies of A indexed by elements of $B$.
Thus, $fun(B,A)$ is a function $B \rightarrow  A$ with the conventional multiplication and finite supports.
The extension of $fun(B, A)$ by $B$ is called the discrete wreath product of
$A,  B$.
Thus, $A \wr  B: = fun(B, A) \leftthreetimes B$ moreover, $bfb^{-1} = f^b$, $b \in B$, $f \in fun(B, A)$.
As well known that a wreath product of permutation groups is associative construction.

Let $G$ be a group acting (from the left) by permutations on
a set $X$ and let $H$ be an arbitrary group.
Then the (permutational) wreath product
$H \wr G$ is the semidirect product $H^X \leftthreetimes G $, % $H \rightthreetimes G $,
 where $G$ acts on the direct power $H^X$ by
the respective permutations of the direct factors.
The group $C_p$ is equipped with a natural action by left shift on $X =\{1,…,p\}$, $p\in \mathbb{N}$.
%The multiplication rule of automorphisms $g$, $h$ which presented in form of wreath recursion \cite{Ne}
%$g=(g\rest[1],g\rest[2],\ldots,g\rest[d])\pi_g, \
%h=(h\rest[1],h\rest[2],\ldots,h\rest[d])\pi_h,$ is given by the formula:
%$$g\cdot h=(g\rest[1]h\rest[\pi_g(1)],g\rest[2]h\rest[\pi_g(2)],\ldots,g\rest[d]h\rest[\pi_g(d)])\pi_g\pi_h.$$

The multiplication rule of automorphisms $g$, $h$ which presented in form of wreath recursion \cite{Ne}
$g=(g\rest[1],g\rest[2],\ldots,g\rest[d])\sigma_g, \
h=(h\rest[1],h\rest[2],\ldots,h\rest[d])\sigma_h,$ is given by the formula:
$$g\cdot h=(g\rest[1]h\rest[\sigma_g(1)],g\rest[2]h\rest[\sigma_g(2)],\ldots,g\rest[d]h\rest[\sigma_g(d)])\sigma_g \sigma_h.$$

We define $\sigma$ as $(1,2,\ldots, p)$ where $p$ is defined by context.
% \note{add wreath recursion description. Product of wreath recursions}
%\note{add wreath product of $C_p$}
%\note{define $\sigma$ as $(1,2,\ldots, p)$ where $p$ is defined by context}
%\noteg{We consider $B\wr (C_p,\, X)$, where $X=\{1,..,p\}$, $B'=\{[f,g]\mid f,g\in B\}$, $p\geq 1$.}

We consider $B \wr (C_p,\, X)$, where $X=\{1,..,p\}$, and $B'=\{[f,g]\mid f,g\in B\}$, $p\geq 1$.
If we fix some indexing $\{x_1, x_2, ... , x_m \}$ of set the $X$, then an element $h\in H^X$ can be written as  $(h_1, ..., h_m ) $ for $h_i \in H$.

 The set $X^*$ is naturally a vertex set of a regular rooted tree, i.e. a connected graph without cycles
and a designated vertex $v_0$ called the root, in which two words are connected by an edge if and only if they are of form $v$ and $vx$, where $v\in X^*$, $x\in X$.
The set $X^n \subset X^*$ is called the $n$-th level of the tree $X^*$
and $X^0 = \{v_0\}$. We denote by $v_{j,i}$ the vertex of $X^j$, which has the number $i$.
Note that the unique vertex $v_{k,i}$ corresponds to the unique word $v$ in alphabet $X$.
For every automorphism $g\in Aut{{X}^{*}}$ and every word $v \in X^{*}$  define the section (state) $g_{(v)} \in AutX^{*}$ of $g$ at $v$ by the rule: $g_{(v)}(x) = y$ for $x, y \in X^*$  if and only if $g(vx) = g(v)y$.
The subtree of $X^{*}$ induced by the set of vertices $\cup_{i=0}^k X^i$ is denoted by $X^{[k]}$.
 The restriction of the action of an automorphism $g\in AutX^*$ to the subtree $X^{[l]}$ is denoted by $g_{(v)}|_{X^{[l]}}$.
 A restriction $g_{(v)}|_{X^{[1]}} $ is called the vertex permutation (v.p.) of $g$ in a vertex $v$.
We call the endomorphism $\alpha|_{v} $ restriction of $g$ in a vertex $v$ \cite{Ne}. For example, if $|X| = 2$ then we just have to distinguish active vertices, i.e., the
vertices for which $\alpha|_{v} $ is non-trivial.
 As well known if $X=\{0, 1 \} $ then $AutX^{[k-1]} \simeq \underbrace {C_2 \wr ...\wr C_2}_{k-1}$ \cite{Ne}.

Let us label every vertex of ${{X}^{l}},\,\,\,0\le l<k$ by sign 0 or 1 in relation to state of v.p. in it. Let us denote state value of $\alpha $ in $v_{ki}$ as ${{s}_{ki}}(\alpha )$ we put that ${{s}_{ki}}(\alpha )=1$ if $\alpha|_{v_{ki}}  $  is non-trivial, %${{\alpha }_{({{v}_{kl}})}}(x)=y,\,\,\,x\ne y$,
 and $s_{ki}(\alpha )=0$ if $\alpha|_{v_{ki}} $ is trivial.
Obtained by such way a vertex-labeled regular tree is an element of $Aut{{X}^{[k]}}$.
All undeclared terms are from \cite{Sam, Gr}.

Let us make some notations.
 The commutator of two group elements $a$ and $b$, denoted
\[
[a,b] = aba\m b\m,
\]
 conjugation by an element $b$ as
%$[a,b]$ is $aba\m b\m$,x
\[
a^b = bab\m,
\]
$\sigma = (1,2, \ldots, p)$. Also $G_k \simeq Syl_2 A_{2^k}$, $B_k = \wr_{i=1}^k C_2 $. The structure of $G_k$ was investigated in \cite{SkAr}. For this research we can regard $G_k$ and $B_k$ as recursively constructed i.e.
$B_1 = C_2$,
$B_k = B_{k-1} \wr C_2$   for $k>1$,
$G_1 = \langle e \rangle$,
$G_k = \{(g_1, g_2)\pi \in B_{k} \mid g_1g_2 \in G_{k-1} \}$ for $k>1$.

%$a^b = bab\m$,
%Let $X_1=\{v_{k-1,1}, v_{k-1,2},..., v_{k-1,2^{k-2}} \} $ and $X_2=\{v_{k-1,2^{k-2}+1}, ..., v_{k-1,2^{k-1}} \}$.

The commutator length of an element $g$ of the derived
subgroup of a group $G$, denoted \emph{clG(g)}, is the minimal $n$ such that there
exist elements $x_1, . . . , x_n, y_1, . . . , y_n$ in G such that $g = [x_1, y_1] . . . [x_n, y_n]$.
The commutator length of the identity element is 0. The commutator width
of a group $G$, denoted $cw(G)$, is the maximum of the commutator lengths
of the elements of its derived subgroup $[G,G]$.
  \end{section}

  \begin{section}{Main result}
We are going to prove that
the set of all commutators $K$ of Sylow 2-subgroup $Syl_{2} A_{{2^k}}$ of the alternating  group ${A}_{2^k}$ is the commutant of $Syl_2 {A_{{2^{k}}}}$.

%\begin{proper}
%\label{G_k_comm_criteria}
%$(g_1, g_2) \in G_k'$ iff $g_1,g_2 \in G_{k-1}$ and $g_1g_2\in F_{k-1}'$.
%\end{proper}

%We will use following theorem which is in \cite{Meld}. Our prove of this theorem is
%more short and we deduce special form of commutator elements.

%We use the following result (one can see this in the book \cite{Meld}).
The following Lemma follows from the corollary 4.9 of the Meldrum's book \cite{Meld}.
\begin{lemma} \label{form of comm} An element of form
$(r_1, \ldots, r_{p-1}, r_p) \in W'= (B \wr C_p)'$ iff product of all $r_i$ (in any order) belongs to $B'$, where $B$ is an arbitrary group.
\end{lemma} % \cite{Meld}

\begin{proof}
%This Lemma follow the corollary 4.9 of the Meldrum's book \cite{Meld}.
%According to the Corollary 4.9 of the Meldrum's book \cite{Meld} we have the following wreath recursion

Analogously to the Corollary 4.9 of the Meldrum's book \cite{Meld} we can deduce new presentation of commutators in form of wreath recursion
\begin{eqnarray*}
w=(r_1, r_2, \ldots, r_{p-1},  r_p),
\end{eqnarray*}
where $r_i\in B$.
If we multiply elements from a tuple $(r_1, \ldots, r_{p-1}, r_p)$, where $r_i={{h}_{i}}{{g}_{a(i)}}h_{ab(i)}^{-1}g_{ab{{a}^{-1}}(i)}^{-1}$, $h, \, g \in B$ and $a,b \in C_p$, then we get a product
\begin{equation} \label{Meld} %\label{H}
 x=\stackrel{p}{ \underset{\text{\it i=1}} \prod} r_i = \prod\limits_{i=1}^{p}{{{h}_{i}}{{g}_{a(i)}}h_{ab(i)}^{-1}g_{ab{{a}^{-1}}(i)}^{-1} \in B'},
\end{equation}
where $x $ is a product of corespondent commutators.
Therefore we can write $r_p = r_{p-1}\m \ldots r_1\m x$. We can rewrite element $x\in B'$ as the product $x = \prod \limits ^{cw(B)}_{j=1} [f_j,g_j]$.

Note that we impose more weak condition on product of all $r_i$ to belongs to $B'$ then in Definition 4.5. of form $P(L)$ in \cite{Meld}.

 In more detail deducing of our representation constructing can be reported in following way.
 If we multiply elements having form of a tuple $(r_1, \ldots, r_{p-1}, r_p)$, where $r_i={{h}_{i}}{{g}_{a(i)}}h_{ab(i)}^{-1}g_{ab{{a}^{-1}}(i)}^{-1}$, $h, \, g \in B$ and $a,b \in C_p$, then in case $cw(B)=0$ we obtain a product
\begin{equation}\label{Meld}
 \stackrel{p}{ \underset{\text{\it i=1}} \prod} r_i = \prod\limits_{i=1}^{p}{{{h}_{i}}{{g}_{a(i)}}h_{ab(i)}^{-1}g_{ab{{a}^{-1}}(i)}^{-1} \in B'}.
\end{equation}

Note that if we rearange elements in (1) as $h_{1} h_{1}^{-1} g_{1}g_2^{-1}h_{2} h_{2}^{-1} g_{1}g_2^{-1} ...  h_{p} h_{p}^{-1} g_{p}g_p^{-1}$ then by reason of such permutations we obtain a product of corespondent commutators. %$\prod\limits_{i=1}^{p-1}(h_{i} h_{i}^{-1} g_{i}g_i^{-1})x \in B'$,
%$\prod\limits_{i=1}^{p-1}(h_{i} h_{i}^{-1} g_{i}g_i^{-1})x \in B'$
% = a_ib_j^{-1}a_{m}b_{l}^{-1}.
Therefore, following equality holds true

\begin{equation}\label{HH}
\prod\limits_{i=1}^{p}{{{h}_{i}}{{g}_{a(i)}}h_{ab(i)}^{-1}g_{ab{{a}^{-1}}(i)}^{-1} } =\prod\limits_{i=1}^{p}h_{i} h_{i}^{-1} g_{i}g_i^{-1}x \in B',
\end{equation}
where $x $ is a product of corespondent commutators.
Therefore,
\begin{eqnarray} \label{form}
(r_1, \ldots, r_{p-1}, r_p) \in W' \mbox{ iff } r_{p-1} \cdot \ldots \cdot r_{1} \cdot r_p = x\in B'
\end{eqnarray}
 Thus, one of elements from coordinate of wreath recursion $(r_1, \ldots, r_{p-1}, r_p) $ depends on rest of $r_i$. This dependence contribute that the product $\prod\limits_{j=1}^{p}r_{j}$ for arbitrary sequence $\{ r_{j} \}_{j=1} ^{p}$
%$r_{p-1} \cdot \ldots \cdot r_{1} \cdot r_p $
 belongs to  $B'$. Thus, $r_p$ can be expressed as:
\begin{eqnarray*}
r_p = r_{1}\m \cdot \ldots \cdot r_{p-1}\m x.
\end{eqnarray*}

Denote a $j$-th tuple, which consists of elements of a wreath recursion, by $(r_{{j}_1},r_{{j}_2},..., r_{{j}_p} )$.
Closedness by multiplication of the set of forms $(r_1, \ldots, r_{p-1}, r_p) \in W= (B \wr C_p)'$
% of elements from a set $K$ of all commutators
  follows from
  %$\prod\limits_{j=1}^{k}  ( r_{j_1} \ldots r_{j_{p-1}} r_{j_p})= \prod\limits_{j=1}^{k} \prod\limits_{i=1}^{p}  r_{j_i} =\prod\limits_{i=1}^{p}  r_{j_i} \prod\limits_{i=1}^{p}  r_{i_2} ... \prod\limits_{i=1}^{p}  r_{i_k} = R_1 R_2 ...  R_{k} \in B '$,

  %                        VERSION OF RECORD 2

\begin{eqnarray} \label{prod}
  \prod\limits_{j=1}^{k}  ( r_{j1} \ldots r_{j{p-1}} r_{jp})= \prod\limits_{j=1}^{k} \prod\limits_{i=1}^{p}  r_{j_i} =  R_1 R_2 ...  R_{k} \in B ',
\end{eqnarray}

  where $r_{ji}$ is $i$-th element from tuple number $j$,  $R_j = \prod\limits_{i=1}^{p}  r_{ji}, \,\, \, 1 \leq j \leq  k$. As it was shown above $R_j = \prod\limits_{i=1}^{p-1}  r_{ji} \in B'$. Therefore, the product (\ref{prod}) of $R_j$, $j \in \{1,...,k \}$ which is similar to the product mentioned in \cite{Meld}, has the property $R_1 R_2 ...  R_{k} \in B '$ too, because of $B '$ is subgroup.
   Thus, we get a product of form (\ref{Meld}) and the similar reasoning as above are applicable.
%   $r_{p-1} \cdot \ldots \cdot r_{1} \cdot r_p$
% transitivity of action $C_p$ on $X$ an

Let us prove the sufficiency condition. % Wise versa
If the set $K$ of elements that satisfy the condition of this theorem that all products of all $r_i$, where every $i$ occurs in this forms once, belong to $B'$, then using the elements of form

 $(r_{1},e,..., e, r_{1}^{-1} )$, ... , $(e,e,...,e, r_{i}, e, r_{i}^{-1} )$, ... ,$(e,e,..., e, r_{p-1}, r_{p-1}^{-1})$, $(e,e,..., e, r_1 r_2 \cdot ...\cdot r_{p-1} )$

  we can express any element of form $(r_1, \ldots, r_{p-1}, r_p) \in W= (C_p \wr B)'$. We need to prove that in such way we can express all element from $W$ and only elements of $W$. The fact that all elements can be generated by elements of $K$ follows from randomness of choice every $r_i$, $i<p$ and the fact that equality (1) holds so construction of $r_p$ is determined.
%Hence, following equality holds true
%\begin{equation}\label{HH}
%\prod\limits_{i=1}^{p}{{{h}_{i}}{{g}_{a(i)}}h_{ab(i)}^{-1}g_{ab{{a}^{-1}}(i)}^{-1} } =\prod\limits_{i=1}^{p}h_{i} h_{i}^{-1} g_{i}g_i^{-1}x \in B',
%\end{equation}
\end{proof}
%where $r_i\in B$ and $r_1 r_2 \ldots r_{p-1} r_p = x \in B'$. Therefore we can write $r_p = r_{p-1}\m \ldots r_1\m x$. We also can rewrite element $x\in B'$ as product of commutators $x = \prod \limits ^{cw(B)}_{j=1} [f_j,g_j]$.

\begin{lemma} \label{form of comm_2} For any group $B$ and integer $p\geq 2, \, p \in \mathbb{N} $ if $w\in (B \wr C_p)'$ then $w$ can be represented as the following wreath recursion
\begin{align*}
w=(r_1, r_2, \ldots, r_{p-1},  r_1\m \ldots r_{p-1}\m \prod \limits ^{k}_{j=1} [f_j,g_j]),
\end{align*}
where $r_1, \ldots, r_{p-1}, f_j, g_j \in B$, and $k\leq cw(B)$.
\end{lemma}
\begin{proof}
%This Lemma follow the corollary 4.9 of the Meldrum's book \cite{Meld}.
%According to the corollary 4.9 of the Meldrum's book \cite{Meld} we have the following wreath recursion
According to the Lemma~\ref{form of comm} we have the following wreath recursion
\begin{align*}
w=(r_1, r_2, \ldots, r_{p-1},  r_p),
\end{align*}
where $r_i\in B$ and $r_{p-1} r_{p-2} \ldots r_2 r_1  r_p = x \in B'$. Therefore we can write $r_p = r_1\m \ldots r_{p-1}\m x$. We also can rewrite element $x\in B'$ as product of commutators $x = \prod \limits ^{k}_{j=1} [f_j,g_j]$ where $k\leq cw(B)$.
%Analogously to the Corollary 4.9 of the Meldrum's book \cite{Meld} we can make new presentation of commutators in form of wreath recursion
%\begin{align*}
%w=(r_1, r_2, \ldots, r_{p-1},  r_p),
%\end{align*}
%where $r_i\in B$ and $r_{p-1} r_{p-2} \ldots r_2 r_1  r_p = x \in B'$. Therefore we can write $r_p = r_1\m \ldots r_{p-1}\m x$. We also can rewrite element $x\in B'$ as product of commutators $x = \prod \limits ^{cw(B)}_{j=1} [f_j,g_j]$.
\end{proof}
\begin{lemma} \label{c_p_wr_b_elem_repr}
For any group $B$ and integer $p\geq 2, \, p \in \mathbb{N}$ if $w\in B \wr C_p$ is defined by the following wreath recursion
\begin{align*}
w=(r_1, r_2, \ldots, r_{p-1},  r_1\m \ldots r_{p-1}\m [f,g]),
\end{align*}
where $r_1, \ldots, r_{p-1}, f_j, g_j \in B$ then we can represent $w$ as commutator
\begin{align*}
w = [(a_{1,1},\ldots, a_{1,p})\sigma, (a_{2,1},\ldots, a_{2,p})],
\end{align*}
where
\begin{align*}
a_{1,i} &=  e \mbox{ for $1\leq i \leq p-1$ },\\
a_{2,1} &= (f\m)^{r_1\m \ldots r_{p-1}\m},\\
a_{2,i} &= r_{i-1} a_{2,i-1}\mbox{ for $2\leq i \leq p$},\\
a_{1,p} &= g^{a_{2,p}\m}.
\end{align*}
\end{lemma}
\begin{proof}
%We consider the following elements
%\begin{align*}
%(a_{1,1},\ldots, a_{1,p})\sigma \in C_p\wr B \mbox{ and } (a_{2,1},\ldots, a_{2,p}) \in C_p\wr B.
%\end{align*}
%We consider the commutator of these elements
Let us to consider the following commutator
\begin{align*}
\kappa &= (a_{1,1},\ldots, a_{1,p})\sigma \cdot (a_{2,1},\ldots, a_{2,p}) \cdot (a_{1,p}\m,a_{1,1}\m,\ldots, a_{1,p-1}\m)\sigma\m \cdot (a_{2,1}\m,\ldots, a_{2,p}\m)\\
&= (a_{3,1}, \ldots, a_{3,p}),
%=(\ldots, a_{1,i}a_{2,1 + (i \bmod p)}a_{1,i}\m a_{2,i}\m, \ldots)
\end{align*}
where
\begin{align*}
a_{3,i} = a_{1,i}a_{2,1 + (i \bmod p)}a_{1,i}\m a_{2,i}\m.
\end{align*}
At first we compute the following
\begin{align*}
a_{3,i} = a_{1,i}a_{2,i+1}a_{1,i}\m a_{2,i}\m = a_{2,i+1} a_{2,i}\m = r_{i} a_{2,i} a_{2,i}\m=  r_i, \mbox{ for $1\leq i \leq p-1$}.
\end{align*}
Then we make some transformation of $a_{3,p}$:
\begin{align*}
a_{3,p}&=a_{1,p}a_{2,1}a_{1,p}\m a_{2,p}\m\\
&=(a_{2,1} a_{2,1}\m) a_{1,p}a_{2,1}a_{1,p}\m a_{2,p}\m\\
&=a_{2,1} [a_{2,1}\m, a_{1,p}] a_{2,p}\m\\
&=a_{2,1}a_{2,p}\m a_{2,p} [a_{2,1}\m, a_{1,p}] a_{2,p}\m\\
%&=(a_{2,p} a_{2,1}\m)\m a_{2,p} [a_{2,1}\m, a_{1,p}] a_{2,p}\m\\
&= (a_{2,p} a_{2,1}\m)\m  [(a_{2,1}\m)^{a_{2,p}}, a_{1,p}^{a_{2,p}}]\\
&= (a_{2,p} a_{2,1}\m)\m  [(a_{2,1}\m)^{a_{2,p} a_{2,1}\m}, a_{1,p}^{a_{2,p}}].
\end{align*}

We transform commutator $\kappa$ in such way that it is similar to the form of $w$. This gives us equations with unknown variables $a_{i,j}$:
\begin{eqnarray*}
\left\{
\begin{matrix}
a_{1,i}a_{2,i+1}a_{1,i}\m a_{2,i}\m &=& r_i, \mbox{ for $1\leq i \leq p-1$},\\
(a_{2,p} a_{2,1}\m)\m &=& r_1\m \ldots r_{p-1}\m,\\
(a_{2,1}\m)^{a_{2,p} a_{2,1}\m} &=& f,\\
a_{1,p}^{a_{2,p}} &=& g.
\end{matrix}\right.
\end{eqnarray*}
In order to prove required statement it is enough to find at least one solution of equations. We set the following
\begin{eqnarray*}
a_{1,i} &=  e \mbox{ for $1\leq i \leq p-1$ }.
\end{eqnarray*}
Then we have
\begin{eqnarray*}
\left\{
\begin{matrix}
a_{2,i+1} a_{2,i}\m &=& r_i, \mbox{ for $1\leq i \leq p-1$},\\
(a_{2,p} a_{2,1}\m)\m &=& r_1\m \ldots r_{p-1}\m,\\
(a_{2,1}\m)^{a_{2,p} a_{2,1}\m} &=& f,\\
a_{1,p}^{a_{2,p}} &=& g.
\end{matrix}\right.
\end{eqnarray*}
Now we can see that the form of the commutator $\kappa$ is similar to the form of $w$.

Let us make the following notation
\begin{align*}
%r' = r_1\m \ldots r_{p-1}\m.
r' = r_{p-1} \ldots r_1.
\end{align*}
We note that from the definition of $a_{2, i}$ for $2\leq i \leq p$ it follows that
\begin{align*}
r_i = a_{2,i+1} a_{2,i}\m, \mbox{ for $1\leq i \leq p-1$}.
\end{align*}
Therefore
\begin{align*}
r' &=  (a_{2,p} a_{2,p-1}\m) (a_{2,p-1} a_{2,p-2}\m)\ldots (a_{2,3} a_{2,2}\m) (a_{2,2} a_{2,1}\m)\\
&= a_{2,p} a_{2,1}\m.
\end{align*}
And then
\[
(a_{2,p} a_{2,1}\m)\m = (r')\m = r_1\m \ldots r_{p-1}\m.
\]
Finally let us to compute the following
\begin{align*}
%a_{3,i} = a_{1,i}a_{2,i+1}a_{1,i}\m a_{2,i}\m = a_{2,i+1} a_{2,i}\m = r_{i} a_{2,i} a_{2,i}\m=  r_i, \mbox{ for $1\leq i \leq p-1$},\\
%(a_{2,p} a_{2,1}\m)\m = r_1\m \ldots r_{p-1}\m,\\
(a_{2,1}\m)^{a_{2,p} a_{2,1}\m} = (((f\m)^{r_1\m \ldots r_{p-1}\m})\m)^{r'} = (f^{(r')\m})^{r'}  = f,\\
a_{1,p}^{a_{2,p}} = (g^{a_{2,p}\m})^{a_{2,p}} = g.
\end{align*}
And now we conclude that
\begin{align*}
a_{3,p} = r_1\m \ldots r_{p-1}\m [f,g].
\end{align*}
Thus, the commutator $\kappa$ is presented exactly in the similar form as $w$ has.
\end{proof}
For future use we formulate previous lemma for the case $p=2$
\begin{corollary} \label{c_2_wr_b_elem_repr}
If $B$ is any group and $w\in B \wr C_2$ is defined by the following wreath recursion
\begin{align*}
w=(r_1,  r_1\m [f,g]),
\end{align*}
where $r_1, f, g \in B$, then $w$ can be represent as commutator
\begin{align*}
w = [(e, a_{1,2})\sigma, (a_{2,1}, a_{2,2})],
\end{align*}
where
\begin{align*}
a_{2,1} &= (f\m)^{r_1\m},\\
a_{2,2} &= r_{1} a_{2,1},\\
a_{1,2} &= g^{a_{2,2}\m}.
\end{align*}
\end{corollary}

\begin{lemma} \label{comm B_k}
For any group $B$ and integer $p\geq 2$ inequality
\begin{align*}
cw(B\wr C_p) \leq \max(1,cw(B))
\end{align*}
holds.
\end{lemma}
\begin{proof}
We can represent any $w\in (B \wr C_p)'$ by Lemma~\ref{form of comm} with the following wreath recursion
\begin{align*}
w&=(r_1, r_2, \ldots, r_{p-1},  r_1\m \ldots, r_{p-1}\m  \prod \limits_{j=1}^{k} [f_{j},g_{j}])\\
 &= (r_1, r_2, \ldots, r_{p-1},  r_1\m \ldots, r_{p-1}\m  [f_{1},g_{1}]) \cdot \prod \limits_{j=2}^{k} [(e, \ldots, e, f_j), (e, \ldots, e, g_j)],
\end{align*}
where $r_1, \ldots, r_{p-1}, f_j, g_j \in B$, $k \leq cw(B)$. Now by the Lemma~\ref{c_p_wr_b_elem_repr}  we  can see that $w$ can be represented as product of $\max(1, cw(B))$ commutators.
\end{proof}
%\note{todo}

\begin{corollary}
If  $W = C_{p_k} \wr \ldots \wr C_{p_1}$ then for $k\geq 2$
$cw(W) =1$.
%\note{wreath product of cyclic groups and inverse limits}
\end{corollary}
\begin{proof}
If $W= C_{p_k} \wr C_{p_{k-1}}$ then according to Lemma \ref{comm B_k} implies that $cw(C_{p_k} \wr C_{p_{k-1}})=1$, because $C_{p_k} \wr C_{p_{k-1}}$ is not commutative group so $cw(W)>0$. If $W= C_{p_k} \wr C_{p_{k-1}} \wr C_{p_{k-2}}$ then according to the inequality  $cw(C_{p_k} \wr C_{p_{k-1}} \wr C_{p_{k-2}}) \leq \max(1,cw(B))$ from Lemma \ref{comm B_k} we obtain $cw(W)=1$. Analogously if $W = C_{p_k} \wr \ldots \wr C_{p_1}$ and supposition of induction for $C_{p_{k}} \wr \ldots \wr C_{p_2}$ holds then using that permutational wreath product is associative construction we obtain from the inequality of Lemma \ref{comm B_k} and $cw( C_{p_k} \wr \ldots \wr C_{p_2})=1$ that $cw(W)=1$.
\end{proof}

\begin{corollary} \label{cw_syl_p_s_p_k_eq_1_and_syl_p_a_p_k_eq_1}  Commutator width  $cw(Syl_p(S_{p^k})) = 1$ for prime $p$ and $k>1$ and commutator width $cw(Syl_p (A_{p^k})) = 1$ for prime $p>2$ and $k>1$.
\end{corollary}
\begin{proof}
Since $Syl_p(S_{p^k}) \simeq  \stackrel{k}{ \underset{\text{\it i=1}}{\wr }}C_p$  (see \cite{Kal, Paw} %\note{is it correct?})
 then $cw(Syl_p(S_{p^k}))=1$. As well known in case $p>2$ $Syl_p S_{p^k} \simeq Syl_p A_{p^k} $ (see \cite{Dm}, then $cw(Syl_p(A_{p^k}))=1$. %\note{is it correct?})
\end{proof}

\begin {definition}
Let us call the index of automorphism $\alpha$ on $X^l$ a number of active v.p. of $\alpha$ on $X^l$, denote it by $In_l(\alpha)$.
\end {definition}

%\begin {definition} \note{alternative definition}
%Let us call the index of automorphism $\alpha \in \{0, 1\}$ on $X^l$ a number module $2$ of active v.p. of %$\alpha$ on the level $X^l$. Denote it by $In_l(l)$.
%\end {definition}

The following Lemma gives us a criteria when the element from the group $Syl_2 S_{2^k}$ belong to $(Syl_2 S_{2^k})'$. %This method based on checking the presentation by portraits of automorphism of binary restricted tree. In detail, our algorithm check the indexes of mentioned automorphism on parity.

%Recall that $  \stackrel{k}{ \underset{\text{\it i=1}}{\wr }}C_p  \simeq Syl_p S_{p^k} $.
\begin{lemma} \label{comm B_k old} An element $g \in B_k$ belongs to commutator subgroup $B'_k$ iff $g$ has even index on ${{X}^{l}}$ for all $0\leq l < k$.
%Commutators of all elements from $B_k$ acting on ${{X}^{[k]}}$
% have all possible even indexes
%  on ${{X}^{l}},\,\,l<k$.
%The set of all commutators $K$ of Sylow 2-subgroup $Syl_{2} A_{{2^k}}$ of the alternating  group ${A}_{2^k}$ is the commutant of $Syl_2 {A_{{2^{k}}}}$.
%A commutator of elements from $Sy{{l}_{2}}{{A}_{{{2}^{k}}}}$ is an element
%  possessing all possible even indexes value
% on ${{X}^{l}},\,\,l<k-1$ and all possible even indexes on ${{X}^{k-1}}$ in intersection with %${{v}_{11}}{{X}^{[k-1]}}$ and with  ${{v}_{12}}{{X}^{[k-1]}}$.
 \end{lemma}

\begin{proof}
Let us prove the ampleness by induction by a number of level $l$ and index of $g$ on $X^l$.
 We first show that our statement for base of the induction is true. Actually, if $\alpha, \beta \in B_{0}$ then $(\alpha \beta {{\alpha }^{-1}}) \beta^{-1}$ determine a trivial v.p. on $X^{0}$. If $\alpha, \beta\in B_{1}$ and $\beta$ has an odd index on $X^1$, then $(\alpha \beta {{\alpha }^{-1}})$ and $\beta^{-1}$ have the same index on $X^1$. Consequently, in this case an index of the product $(\alpha \beta {{\alpha }^{-1}}) \beta^{-1}$ can be 0 or 2. Case where $\alpha, \beta \in B_{1}$ and has even index on $X^1$, needs no proof, because the product and the sum of even numbers is an even number.

To finish the proof it suffices to assume that for $B_{l-1}$ statement holds and prove that it holds for $B_l$.
Let $\alpha, \beta$ are an arbitrary automorphisms from $Aut X^{[k]}$ and $\beta$ has index $x$ on $X^{l}, \, l<k$, where $0 \leq x \leq 2^{l} $.% , then conjugation $(\alpha \beta {{\alpha }^{-1}})$ acts by arbitrary permutation of active vertices that has $\beta$ on $X^l$.
   A conjugation of an automorphism $\beta $  by arbitrary $\alpha \in Aut{{X}^{[k]}}$ gives us arbitrary permutations of $X^l$ where $\beta $ has active v.p.

 Thus following product $(\alpha \beta {{\alpha }^{-1}}) \beta ^{-1}$ admits all possible even indexes on $X^l,  l<k$ from 0 to $2x$. In addition $[\alpha, \beta]$ can has arbitrary permitted assignment (arrangement) of v.p. on $X^l$. Let us present $B_k$ as $B_k=B_l \wr B_{k-l}$, so elements $\alpha, \beta$ can be presented in form of wreath recursion $\alpha = (h_{1},...,h_{2^l })\pi_1, \,  \beta = (f_{1},...,f_{2^l })\pi_2$, $h_{i}, f_{i} \in B_{k-l} ,\  0<i \leq 2^l$ and $h_{i}, f_{j}$ corresponds to sections of automorphism in vertices of $X^l$ of isomorphic subgroup to $B_l$ in $Aut X^{[k]}$.
%Множення автоморфізмів $g$ і $h$, записаних у вигляді
%$$g=(g\rest[1],g\rest[2],\ldots,g\rest[d])\pi_g, \
%h=(h\rest[1],h\rest[2],\ldots,h\rest[d])\pi_h,$$ визначається за
%правилом:
%$$g\cdot h=(g\rest[1]h\rest[\pi_g(1)],g\rest[2]h\rest[\pi_g(2)],\ldots,g\rest[d]h\rest[\pi_g(d)])\pi_g\pi_h.$$
%$$g\cdot h=(g\rest[1]h\rest[\pi_g(1)],g\rest[2]h\rest[\pi_g(2)],\ldots,g\rest[d]h\rest[\pi_g(d)])\pi_g\pi_h.$$
Actually, the parity of this index are formed independently of the action of
$Aut X^{[l]}$ on $X^l$. So this index forms as a result of multiplying of elements of commutator presented as wreath recursion $(\alpha \beta \alpha^{-1}) \cdot \beta ^{-1} = (h_{1},...,h_{2^l})\pi_1 \cdot (f_{1},...,f_{2^l })\pi_2= (h_{1},...,h_{2^l}) (f_{\pi_1(1)},...,f_{\pi_1(2^l)})\pi_1 \pi_2 $, where $h_{i}, f_{j} \in {B}_{k-l}$, $l<k$ and besides automorphisms corresponding to $h_{i}$ are $x$ automorphisms which has active v.p. on $X^l$. Analogous automorphisms $h_{i}$ has number of active v.p. equal to $x$. As a result of multiplication we have automorphism with index $2i:$ $0 \leq 2i \leq 2x$.
%$h_{i1}, h_{i2} \in St_{B_k} (l)$
Consequently, commutator $[\alpha, \beta]$ has arbitrary even indexes on $X^m$, $m<l$ and we showed by induction  that it has even index on $X^l$.

Let us prove this Lemma by induction on level $k$. Let us to suppose that we prove current Lemma (both sufficiency and necessity) for $k-1$. Then we rewrite element $g \in B_k$ with wreath recursion
\[
g = (g_1, g_2)\sigma^i,
\]
where $i \in \{0, 1\}$.

Now  we consider sufficiency. 

Let $g \in B_k$ and $g$ has all even indexes on $X^j$ $0 \leq j < k$ we need to show that $g \in B'_k$. According to condition of this Lemma $g_1 g_2$ has even indexes. An element $g$ has form $g=(g_1, g_2)$, where $g_1, g_2 \in B_{k-1}$, and products $g_1 g_2 = h \in B'_{k-1}$ because $h \in B_{k-1}$ and for $B_{k-1}$ induction assumption holds. Therefore all products of form $g_1 g_2$ indicated in formula \ref{Meld} belongs to $B'_{k-1}$. Hence, from Lemma \ref{form of comm} follows that $g= (g_1, g_2 ) \in B'_k$.
% assumption of mathematical induction for $g_1, g_2 \in B_{k-1}$ holds viz $g_1, g_2 \in B'_{k-1}$. sufficiency
\end{proof}

%\begin{theorem} \label{max}
%\textbf{A maximal 2-subgroup}  of $Aut{{X}^{\left[ k \right]}}$ acting by even permutations on ${{X}^{k}}$ has the structure of the semidirect product $G_k \simeq  B_{k-1} \ltimes W_{k-1} $ and is isomorphic to $Syl_2A_{2^k}$.
%\end{theorem}

An automorphisms group of the subgroup ${C_{2}^{{{2}^{k-1}-1}}}$ is based on permutations of copies of $C_2$. Orders of $\stackrel{k-1}{ \underset{\text{\it i=1}}{\wr } }C_2  $ and ${C_{2}^{{{2}^{k-1}-1}}}$ are equals. A homomorphism from $\stackrel{k-1}{ \underset{\text{\it i=1}}{\wr } }C_2  $ into $Aut({C_{2}^{{{2}^{k-1}-1}}})$ is injective because a kernel of action $\stackrel{k-1}{ \underset{\text{\it i=1}}{\wr } }C_2  $ on ${C_{2}^{{{2}^{k-1}-1}}}$ is trivial, action is effective. The group $G_k$ is a proper subgroup of index 2 in the group $\stackrel{k}{ \underset{\text{\it i=1}}{\wr } }C_2  $ \cite{Dm, SkAr, Sk }.
%ПРАВКИ
The following theorem can be used for proving structural property of Sylow subgroups.
\begin{theorem} \label{max}
\textbf{A maximal 2-subgroup}  of $Aut{{X}^{\left[ k \right]}}$ acting by even permutations on ${{X}^{k}}$ has the structure of the semidirect product $G_k \simeq  B_{k-1} \ltimes W_{k-1} $ and is isomorphic to $Syl_2A_{2^k}$. Also $G_k < B_k$.
\end{theorem}
%ПРАВКИ
The prove of this theorem is in \cite{SkAr}.

An even easier Proposition, that needs no proof, is the following.
\begin{proposition}\label{B_k_criteria}
An element $(g_1, g_2)\sigma^i$, $i \in \{0, 1\} $ of wreath power $\stackrel{k}{ \underset{\text{\it i=1}}{\wr } }C_2  $ belongs to its subgroup $ G_k$, iff $g_1g_2 \in G_{k-1}$.
%$(g_1, g_2)\sigma^i \in G_k$
\end{proposition}
\begin{proof}
%content...
This fact follows from the structure of elements of $G_k$ described Theorem \ref{max} in and the construction of wreath recursion. Indeed, due to  structure of elements of $G_k$ described in Theorem \label{max}, we have action on $X^k$ by an even permutations because subgroup $W_{k-1}$, containing even number of transposition, acts on $X^k$ only by even permutation. The condition $g_1g_2 \in G_{k-1}$ is equivalent to index of $g$ on $X^{k-1}$ is even but this condition equivalent to condition that $g$ acting on $X^k$ by even permutation.
\end{proof}

\begin{lemma}\label{L_k_comm_criteria}
An element $(g_1, g_2)\sigma^i \in G_k'$ iff $g_1,g_2 \in G_{k-1}$ and $g_1g_2\in B_{k-1}'$.
\end{lemma}

\begin{proof}
Indeed, if $(g_1, g_2) \in G_k'$ then indexes of $g_1$ and $g_2$ on $X^{k-1}$ are even according to Lemma \ref{comm} thus, $g_1,g_2 \in G_{k-1}$. A sum of indexes of $g_1$ and $g_2$ on $X^{l}$, $l<k-1$ are even according to Lemma \ref{comm} too, so index of product $g_1 g_2$ on $X^{l}$ is even. Thus, $g_1g_2\in B_{k-1}'$. Hence, necessity is proved.
%Hence, necessity for $(g_1, g_2) \in G_k'$ is proved.

Let us prove the sufficiency via Lemma \ref{comm}.
Wise versa, if $g_1,g_2 \in G_{k-1}$ then indexes of these automorphisms on $X^{k-2}$ of subtrees $v_{11}X^{[k-1]}$ and $v_{12}X^{[k-1]}$ are even as elements from $ G_k'$ have.  %Automorphism $g_1$ is restriction of automorphism of $X^{[k]}$ on subtree $v_{11}X^{[k]}$.
 The product $g_1g_2$ belongs to $B_{k-1}'$ by condition of this Lemma so sum of indexes of $g_1, g_2$ on any level $X^l$,  $0 \leq l<k-1 $ is even. Thus, the characteristic properties of $G_k'$ described in Lemma \ref{comm} holds.
\end{proof}

\begin{proposition}\label{comm_F_k_is_subgroup_of_L_k}
The following inclusion $B_k'<G_k$ holds.
\end{proposition}
 \begin{proof}
 Indeed,  $B_k' = \wr_{i=1}^{k-1} C_2 = B_{k-1}$ and as we define $G_k \simeq  B_{k-1} \ltimes W_{k-1} $ so $B_k' < G_k$.
 \end{proof}
\begin{proposition}\label{G_k_is_normal_in_B_k}
The group $G_k$ is normal in wreath product $\stackrel{k}{ \underset{\text{\it i=1}}{\wr } }C_2$ i.e. $ G_k \lhd B_k$.
\end{proposition}
\begin{proof}
The commutator of $B_k$ is $B_k' < B_{k-1}$. In other hand $ B_{k-1} <  G_{k}$ because $G_k \simeq  B_{k-1} \ltimes W_{k-1} $ consequently $B_k' < G_{k}$. Thus, $G_{k} \lhd B_{k}$.
\end{proof}
There exists a normal embedding  (normal  injective monomorphism)  $\varphi :\,\,{{G}_{k}}\to {{B}_{k}}$  \cite{Heinek} i.e.$~~{{G}_{k}}\triangleleft {{B}_{k}}$. Actually, it implies from  Proposition \ref{G_k_is_normal_in_B_k}. Also according to \cite{SkAr} index $\left| {{B}_{k}}~ :~{{G}_{k}} \right|=2$ so ${G}_k$ is a normal subgroup that is a factor subgroup $~~{}^{{{B}_{k}}}/{}_{{{C}_{2}}}\simeq {{G}_{k}}$.
%  Фразу such that $~~{{G}_{k}}\triangleleft {{B}_{k}}$ лутше в статье закоментироавть ведь это растолкование термина normal embedding.

\begin{theorem}\label{_comm_F_k_eq_[L_k,F_k]}
Elements of $\aut[k]'$ have the following form $\aut[k]'=\{[f,l]\mid f\in B_k, l\in G_k\}=\{[l,f]\mid f\in B_k, l\in G_k\}$.
\end{theorem}
\begin{proof}
It is enough to show either $B_k'=\{[f,l]\mid f\in B_k, l\in G_k\}$ or $B_k'=\{[l,f]\mid f\in B_k, l\in G_k\}$ because if $f = [g,h]$ then $f\m = [h,g]$.

We prove the proposition by induction on $k$.
$B_1' = \langle e \rangle$

We already know \ref{form of comm_2} that every element $w\in\aut[k]'$ we can represent as
\begin{align*}
w=(r_1, r_1\m [f,g])
\end{align*}
for some $r_1,f\in \aut[k-1]$ and $ g\in \syl[k-1]$ (by induction hypothesis). By the Corollary~\ref{c_2_wr_b_elem_repr} we can represent $w$ as commutator of
\begin{align*}
(e,a_{1,2})\sigma \in \aut[k] \mbox{ and } (a_{2,1}, a_{2,2}) \in \aut[k],
\end{align*}
where
\begin{align*}
a_{2,1} &= (f\m)^{r_1\m},\\
a_{2,2} &= r_{1} a_{2,1},\\
a_{1,2} &= g^{a_{2,2}\m}.
\end{align*}
We note that as $g \in G_{k-1}$ then by proposition~\ref{B_k_criteria} we have $(e,a_{1,2})\sigma \in \syl[k]$.
\end{proof}
Directly from this Proposition follows next Corollary, that needs no proof.
\begin{remark}\label{_comm_B_k_eq_[b_k,b_k]}
Let us to note that Theorem~\ref{_comm_F_k_eq_[L_k,F_k]} improve Corollary~\ref{cw_syl_p_s_p_k_eq_1_and_syl_p_a_p_k_eq_1} for the case $p=2$.
%$\syl[k]'=\{[f_1,f_2]\mid f_1\in \syl[k], f_2\in \syl[k]\}$.
%The set $B'_k$ coincides with set of all commutators of $B_k$, put it differently $cw(B_k)=1$.
\end{remark}

%\begin{proposition}
%For all $g\in F_k$: $g^2\in F_k'$
%\end{proposition}
%\begin{proof}
%Induction on $k$: $F_1=C_2$
%\begin{align*}
%g = (g_1, g_2) \sigma^i\\
%g^2 = (g_1^2, g_2^2)\mbox{ or } g^2 = (g_1g_2, g_2g_1)
%\end{align*}
%\begin{align*}
%g_1^2 \in F_{k-1}'\\
%g_2^2 \in F_{k-1}'\\
%g_1^2 g_2^2 \in F_{k-1}'
%\end{align*}
%Then $g^2 = (g_1^2, g_2^2) \in F_{k}'$.
%\begin{align*}
%g_1 g_2^2 g_1 = g_1^2 g_2^2 [g_2^{-2}, g_1\m] \in F_{k-1}'
%\end{align*}
%\end{proof}

\begin{proposition}\label{B'_k and B^2_k}
If $g $ is an element of wreath power $\stackrel{k}{ \underset{\text{\it i=1}}{\wr } }C_2 \simeq B_k $ then $g^2 \in B'_{k}$.
\end{proposition}
\begin{proof}
As it was proved in Lemma \ref{comm B_k old} commutator $[\alpha, \beta]$ from $B_k$ has arbitrary even indexes on $X^m$, $m<k$. Let us show that elements of $B_k^2$ have the same structure.
%Conversely, commutator $[\alpha, \beta]$ has arbitrary even indexes on $X^m$, $m<l$ by assumption of an induction and we showed that it has even index on $X^l$.

Let $\alpha, \beta \in B_k$ an indexes of the automorphisms $\alpha^2 $, $(\alpha \beta)^2 $  on $X^l, \, l<k-1$ are always even. In more detail the indexes of $\alpha^2 $, $(\alpha \beta)^2 $ and $\alpha^{-2} $ on $X^l$ are determined exceptionally by the parity of indexes of $\alpha $ and $\beta $ on ${{X}^{l}}$. Actually, the parity of this index are formed independently of the action of
$Aut X^l$ on $X^l$.
So this index forms as a result of multiplying of elements $\alpha \in B_k$ presented as wreath recursion $ \alpha^2 = (h_{1},...,h_{2^l})\pi_1 \cdot (h_{1},...,h_{2^l })\pi_1= (h_{1},...,h_{2^l}) (h_{\pi_1(1)},...,h_{\pi_1(2^l)})\pi_1 ^2 $, where $h_{i}, h_{j} \in {B}_{k-l}, \, \pi_1 \in B_l$, $l<k$ and besides automorphisms corresponding to $h_{i}$ are $x$ automorphisms which has active v.p. on $X^l$. Analogous automorphisms $h_{i}$ has number of active v.p. equal to $x$. As a result of multiplication we have automorphism with index $2i:$ $0 \leq 2i \leq 2x$.
%of v.p. from the higher levels and this index is even.
%Since an index of $\alpha \beta $ on ${{X}^{l}}$ is an arbitrary $x:$ $0\leq x \leq 2^l$ then an index of $(\alpha \beta)^2 $ is arbitrary even number that is between $0$ and $2^l$.

Since $g^2 $ admits only an even index on $X^l$ of $Aut X^{[k]}$, $0<l<k$, then $g^2 \in B'_{k}$ according to lemma \ref{comm B_k old} about structure of a commutator subgroup.
\end{proof}

Since as well known a group $G_k^2$ contains the subgroup $G'$ then a product $G^2 G'$ contains all elements from the commutant. Therefore, we obtain that $G_k^2 \simeq G'_k$.

%Alternative prove for $G'_k$ is as follows. Here is a NEW our PROVE !!!
\begin{proposition}\label{g_sq_in_G_k}
For arbitrary $g\in G_k$ following inclusion $g^2\in G_k'$ holds.
\end{proposition}
\begin{proof}
Induction on $k$: for $G^2_1$ elements has form $ ((e, e)\sigma)^2 = e $ where $\sigma = (1,2)$ so statement holds. In general case when $k>1$ elements of $G_k$ has form
\begin{align*}
g = (g_1, g_2) \sigma^i,  \, g_1 \in B_{k-1}, \, i \in{0.1} \\
\mbox{ then we have two possibilities} \\
g^2 = (g_1^2, g_2^2)\mbox{ or } g^2 = (g_1g_2, g_2g_1).
\end{align*}

We first show that
\begin{align*}
g_1^2 \in B_{k-1}',
g_2^2 \in B_{k-1}'\\
\mbox{ after we will prove} \\
g_1 g_2 \cdot g_2 g_1 \in B_{k-1}',
\end{align*}
 actually, according to Proposition 14 $g_1^2, g_2^2\in B_{k-1}'$ then $g_1^2 g_2^2 \in B_{k-1}'$ and $ g_1^2, g_2^2 \in G_{k-1}$ by Proposition \ref{comm_F_k_is_subgroup_of_L_k} also $ g_1^2, g_2^2 \in G_{k-1}$ by induction assumption. From Proposition~\ref{B_k_criteria} it follows that $g_1 g_2  \in G_{k-1}$.

% actually, according Proposition \ref{B'_k and B^2_k} $g_1^2, g_1^2\in B_{k-1}'$ then $g_1^2 g_2^2 \in %B_{k-1}'$ and $ g_1^2, g_2^2 \in G_{k-1}$ by Proposition \ref{B_k_criteria} also $ g_1^2, g_2^2 \in %G_{k-1}$ by induction assumption. From Proposition~\ref{B_k_criteria} it follows that $g_1 g_2  \in %G_{k-1}$ also $ g_1^2, g_2^2 \in G_{k-1}$.

Note that $B_{k-1}' < B_{k-2}$. In other hand $ B_{k-2} <  G_{k-1}$ because $G_{k-1} \simeq  B_{k-2} \ltimes W_{k-2} $ consequently $B_{k-1}' < G_{k-1}$. Besides we have $g_1^2 \in B_{k-1}'$ hence $g_1^2 \in G_{k-1}$.

 Thus, we can use Lemma~\ref{L_k_comm_criteria} (about $G'_k$) from which yields  $g^2 = (g_1^2, g_2^2) \in G_{k}'$.

%$\begin{align*}
%g_1g_2\in L_{k-1} \mbox{ by proposition~\ref{B_k_criteria}}\\
%g_2g_1 = g_1g_2 g_2\m g_1\m g_2 g_1 = g_1g_2 [g_2\m, g_1\m]\in L_{k-1}\mbox{ %proposition~\ref{comm_F_k_is_subgroup_of_L_k}}\\
%g_1 g_2^2 g_1 = g_1^2 g_2^2 [g_2^{-2}, g_1\m] \in F_{k-1}'
%\end{align*}
%\begin{align*}
%g_1g_2\in L_{k-1} \mbox{ by proposition~\ref{B_k_criteria}}\\
%g_2g_1 = g_1g_2 g_2\m g_1\m g_2 g_1 = g_1g_2 [g_2\m, g_1\m]\in L_{k-1}\mbox{ %proposition~\ref{comm_F_k_is_subgroup_of_L_k}}\\
%\end{align*}

Consider second case $g^2 = (g_1g_2, g_2g_1)$
\begin{align*}
g_1g_2\in G_{k-1}   \mbox{ by proposition~\ref{B_k_criteria}}\\
g_2g_1 = g_1g_2 g_2\m g_1\m g_2 g_1 = g_1g_2 [g_2\m, g_1\m]\in G_{k-1}\mbox{ by propositions~\ref{comm_F_k_is_subgroup_of_L_k} and~\ref{B_k_criteria}}\\
g_1g_2 \cdot g_2 g_1 = g_1 g_2^2 g_1 = g_1^2 g_2^2 [g_2^{-2}, g_1\m] \in B_{k-1}'
\end{align*}
Note that $g_1^2, g_2^2 \in B'_{k-1}$ according to Proposition \ref{B'_k and B^2_k} this is a reason why $g_1^2 g_2^2 [g_2^{-2}, g_1\m] \in B_{k-1}'$.
Thus, $( g_1g_2, g_1g_2) \in G'_{k}$ by Lemma~~\ref{L_k_comm_criteria}.
\end{proof}

%НАЧАЛО ВСТАВКИ ДОПОЛНИТЕЛЬНОГО МАТЕРИАЛА О КОММУТАНТЕ И ЦЕНТРАЛИЗАТОРЕ

  \end{section}

Let $X_1=\{v_{k-1,1}, v_{k-1,2},..., v_{k-1,2^{k-2}} \} $ and $X_2=\{v_{k-1,2^{k-2}+1}, ..., v_{k-1,2^{k-1}} \}$.

We will call a distance structure $\rho_l(\theta)$ of $\theta $ a tuple of distances between its active vertices from $X^l$.
Let group $Sy{{l}_{2}}{{A}_{{{2}^{k}}}}$ acts on $X^{[k]}$.
\begin{lemma} \label{comm}
An element $g$ belongs to $G_k' \simeq Syl_2{A_{2^k}}$ iff $g$ is arbitrary element from $G_k$ which has all even indexes  on ${{X}^{l}},\,\,l<k-1$ of ${{X}^{[k]}}$ and on ${{X}^{k-2}}$ of subtrees ${{v}_{11}}{{X}^{[k-1]}}$ and ${{v}_{12}}{{X}^{[k-1]}}$.

%Commutators of all elements from $Sy{{l}_{2}}{{A}_{{{2}^{k}}}}$
% have all possible even indexes
%  on ${{X}^{l}},\,\,l<k-1$ of ${{X}^{[k]}}$ and on ${{X}^{k-2}}$ of subtrees ${{v}_{11}}{{X}^{[k-1]}}$ and ${{v}_{12}}{{X}^{[k-1]}}$, $(Sy{{l}_{2}}{{A}_{{{2}^{k}}}})'$ consists of commutators.

%The set of all commutators $K$ of Sylow 2-subgroup $Syl_{2} A_{{2^k}}$ of the alternating  group ${A}_{2^k}$ is the commutant of $Syl_2 {A_{{2^{k}}}}$.

%A commutator of elements from $Sy{{l}_{2}}{{A}_{{{2}^{k}}}}$ is an element
%  possessing all possible even indexes value
% on ${{X}^{l}},\,\,l<k-1$ and all possible even indexes on ${{X}^{k-1}}$ in intersection with %${{v}_{11}}{{X}^{[k-1]}}$ and with  ${{v}_{12}}{{X}^{[k-1]}}$.
 \end{lemma}

\begin{proof}
%As we said any automorphism $\theta$ from $G_k$ has even index on $X_{k-1}$ so $\theta$ has the same parity of numbers of active v.p. on $X_1$ and $X_2$.
%Let us consider ${\alpha }\in Aut{{v}_{11}}{{X}^{\left[ k-1 \right]}}$, conjugation by such ${\alpha }$ of any $\theta \in G_k$ only permute vertices inside $X_1$ ($X_2$).

Let us prove the ampleness by induction on a number of level $l$ and index of automorphism $g$ on $X^l$.
%Recall that any automorphism $\theta \in Syl_2 A_n$ has an even index on $X^{k-1}$ so number parities of active v. p. on ${{X}_{1}}$ and on ${{X}_{2}}$ are the same.
Conjugation by automorphism $\alpha$ from $Aut{{v}_{11}}{{X}^{\left[ k-1 \right]}}$ of automorphism $\theta $, that has index $x:$ $1 \leq x \leq 2^{k-2}$ on ${{X}_{1}}$ does not change $x$. Also automorphism $\theta^{-1} $ has the same number $x$ of v. p. on $X_{k-1}$ as $\theta $ has. If $\alpha$ from $Aut{{v}_{11}}{{X}^{\left[ k-1 \right]}}$ and $ \alpha \notin Aut{{X}^{\left[ k \right]}}$ then conjugation $(\alpha \theta {{\alpha }^{-1}})$ permutes vertices only inside $X_1$ ($X_2$).

 Thus, ${\alpha }\theta {\alpha^{-1} }$ and $\theta$ have the same parities of number of active v.p. on $X_1$ ($X_2$). Hence, a product ${\alpha }\theta {\alpha^{-1} } \theta^{-1}$ has an even number of active v.p. on $X_1$ ($X_2$) in this case. More over  a coordinate-wise sum by \texttt{mod2} %$mod2$
  of active v. p. from $(\alpha \theta {{\alpha }^{-1}})$ and $\theta^{-1}$ on $X_1$ ($X_2$) is even and equal to $y:$ $0 \leq y \leq 2x$.

%Because number parities of active v.p. from $(\alpha \theta {{\alpha }^{-1}})$ and ${{\theta }^{-1}}$ from $X_1$ ($X_2$) are the same.

   If conjugation by $\alpha$ permutes sets $X_1$ and $X_2$ then there are  coordinate-wise sums of no trivial v.p. from $\alpha \theta \alpha^{-1} \theta^{-1}$ on $X_1$ (analogously on $X_2$) have form: \\ $({{s}_{k-1,1}}(\alpha \theta {{\alpha }^{-1}}),..., {{s}_{k-1, 2^{k-2}}}(\alpha \theta {{\alpha }^{-1}}) )\oplus ({{s}_ {k-1,1}}(\theta^{-1}), ..., {{s}_{k-1,{{2}^{k-2}}}}(\theta^{-1} ))$.
  % ,..., {{s}_{{k-1,}}{2^{k-2}+1}}(\alpha \theta {{\alpha }^{-1}}
   This sum has even number of v.p. on $X_1$ and $X_2$ because $(\alpha \theta {{\alpha }^{-1}})$ and ${{\theta }^{-1}}$ have a same parity of no trivial v.p. on $X_1$ ($X_2$).  Hence, $(\alpha \theta {{\alpha }^{-1}}){{\theta }^{-1}}$ has even number of v.p. on ${{X}_{1}}$ as well as on ${{X}_{2}}$.

%On $X^l, \, l< k-1$ commutator $(\alpha \theta {{\alpha }^{-1}}){{\theta }^{-1}}$ has even index because if $\theta $ has odd (even) index then $\theta^{-1}$ and $\alpha \theta {{\alpha }^{-1}}$ have odd (even) indexes so coordinate-wise sum by $mod2$  of active v.p. from $\alpha \theta {{\alpha }^{-1}}$ and ${{\theta }^{-1}}$ is always even.

An automorphism $\theta $ from $G_k$ was arbitrary so number  of active v.p. $x$ on $X_1$ is arbitrary $0 \leq x\leq 2^l$. And ${\alpha }$ is arbitrary from $AutX^{[k-1]}$ so vertices can be permuted in such way that the commutator $[{\alpha },\theta]$ has arbitrary even number $y$ of active v.p. on $X_1$, $0 \leq y \leq 2x$.

 A conjugation of an automorphism $\theta $ having index $x$, $1 \leq x \leq 2^{l}$ on ${{X}^{l}}$  by different $\alpha \in Aut{{X}^{[k]}}$  gives us all tuples of active v.p. with the same $\rho_l(\theta)$ that $\theta $ has on ${{X}^{l}}$, by which $Aut{{X}^{[k]}}$ acts on $X^l$. Let supposition  of induction for element $g$  with index $2k-2$ on $X^l$ holds so $g=(\alpha \theta {{\alpha }^{-1}}){{\theta }^{-1}}$, where $In_l(\theta)=x$. To make a induction step we complete $\theta$ by such active vertex $v_{l,x}$ too it has suitable distance structure for $g=(\alpha \theta {{\alpha }^{-1}}){{\theta }^{-1}}$, also if $g$ has rather different distance structure $d_l(g)$ from $d_l(\theta)$ then have to change $\theta$. In case when we complete $\theta$ by $v_{l,x}$ it has too satisfy a condition $(\alpha \theta {{\alpha }^{-1}}) (v_{x+1})=v_{l, y}$, where $v_{l, y}$ is a new active vertex of g on $X^l$.
  Note that $v(x+1) $ always can be chosen such that acts in such way $\alpha(v(x+1)) = v(2k+2)$ because action  of $\alpha$ is 1-transitive. Second vertex arise when we multiply $(\alpha \theta {{\alpha }^{-1}})$ on $\theta^{-1}$. %Thus $in_l(g)$ became to be $2k+2$.
 %  we complete $\theta$ by one active vertex on $X^l$ let this vertex $v_x+1$
 % and $(\alpha \theta {{\alpha }^{-1}}) (v_{x+1})=v_{2k+1}$
  %Note that $\alpha $ always can be chosen such that acts in such way $\alpha(v(x+1)) = v(2k+2)$ because action is 1-transitive.
  Hence $In_l(\alpha \theta {{\alpha }^{-1}})=2k+2$ and coordinates of new vertices $v_{2k+1}, v_{2k+2}$ are arbitrary from 1 to $2^l$.
%  Note that $\alpha $ always can be chosen such that acts non-trivial on $v(2k-1)$.

%  Choice of these vertices is such too $(\alpha \theta {{\alpha }^{-1}})(v_{x+1})=v_{2k+1}$ and $(\alpha \theta {{\alpha }^{-1}})(v_{x+2})=v_{2k+2}$, where $\alpha$ is the same. Hence $In_l(\alpha \theta {{\alpha }^{-1}}))=2k+2$.

 So multiplication $(\alpha \theta {{\alpha }^{-1}})\theta $ generates a commutator having index $y$  equal to coordinate-wise sum by $mod 2$ of no trivial v.p. from vectors $({{s}_{l1}}(\alpha \theta {{\alpha }^{-1}}),{{s}_{l}}_{2}(\alpha \theta {{\alpha }^{-1}}),...,{{s}_{l{{2}^{l}}}}(\alpha \theta {{\alpha }^{-1}}))\oplus ({{s}_{l1}}(\theta ),{{s}_{l}}_{2}(\theta ),...,{{s}_{l{{2}^{l}}}}(\theta ))$  on ${{X}^{l}}$. A indexes parities of  $\alpha \theta {{\alpha }^{-1}}$  and  ${{\theta }^{-1}}$ are same so their sum by $mod 2$ are even.  Choosing $\theta $ we can  choose an arbitrary index $x$ of $\theta $ also we can choose arbitrary $\alpha $ to make a permutation of active v.p. on ${{X}^{l}}$.  Thus, we obtain an element with arbitrary even index on ${{X}^{l}}$ and arbitrary location of active v.p. on ${{X}^{l}}$.

%Зміни
Check that property of number parity of v.p. on ${{X}_{1}}$  and on ${{X}_{2}}$  is closed with respect to conjugation. We know that numbers of active v. p. on ${{X}_{1}}$ as well as on ${{X}_{2}}$ have the same parities. So
action by conjugation only can permutes it, hence, we again get the same  structure of element. Conjugation by automorphism $\alpha $ from  $Aut{{v}_{11}}{{X}^{\left[ k-1 \right]}}$  automorphism $\theta $, that has odd number of  active v. p. on ${{X}_{1}}$  does not change its parity.
Choosing the $\theta $ we can choose arbitrary index $x$ of
$\theta $ on ${{X}^{k-1}}$ and number of active v.p. on ${{X}_{1}}$  and  ${{X}_{2}}$  also we can choose arbitrary $\alpha $ to make a permutation active v.p. on ${{X}_{1}}$  and  ${{X}_{2}}$. Thus, we can generate all possible elements from a commutant. Also this result follows from Lemmas \ref{L_k_comm_criteria} and \ref{comm B_k old}.

Let us check that the set of all commutators $K$ from $Syl_2 A_{2^k}$ is closed with respect  to multiplication of commutators. Let $\kappa_1, \kappa_2 \in K$ then $\kappa_1 \kappa_2$ has an even index on $X^l$, $l<k-1$ because  coordinate-wise sum $({{s}_{l,1}}(\kappa_1),..., {{s}_{k-1, 2^l}}(\kappa_1) )\oplus ({{s}_ {l,\kappa_1(1)}}(\kappa_2), ..., {{s}_{l,\kappa_1({{2}^{l}})}}(\kappa_2 ))$.
 of two $2^l$-tuples of v.p. with an even number of no trivial coordinate has even number of such coordinate.  Note that conjugation of $\kappa $ can permute sets ${{X}_{1}}$ and ${{X}_{2}}$  so parities of $x_1$ and $X_2$ coincide. It is obviously index of $\alpha \kappa \alpha^{-1}$ is even as well as index of $\kappa $.

Check that a set $K$ is a set  closed with respect  to conjugation.

 Let $\kappa \in K$, then $\alpha \kappa {{\alpha }^{-1}}$  also belongs to $K$, it is so because  conjugation does not change index of an automorphism on a level. Conjugation only  permutes vertices on level because elements of $Aut{{X}^{\left[ l-1 \right]}}$ acts  on vertices of  ${{X}^{l}}$. But as it was proved above elements  of $K$ have all possible indexes on ${{X}^{l}}$, so as a result of conjugation $\alpha \kappa {{\alpha }^{-1}}$ we obtain an element from $K$.

Check that the set of commutators is closed with respect to multiplication of commutators.
Let $\kappa_1, \kappa_2 $ be an arbitrary commutators of $G_k$. The parity of the number of vertex permutations on $X^l$ in the product $\kappa_1 \kappa_2 $  is determined exceptionally by the parity of the numbers of active v.p. on ${{X}^{l}}$ in $\kappa_1$ and $\kappa_2$ (independently from the action of v.p. from the higher levels). Thus $\kappa_1 \kappa_2 $ has an even index on $X^l$.

 Hence, normal closure of the set $K$ coincides with $K$.
\end{proof}
%КОНЕЦ ВРЕМЕННОЙ ВСТАВКИ МОЕЙ ЛЕММЫ О КОММУТАНТЕ
\begin{lemma} \label{comm}
%Commutators  of all elements from $G_k$
% have all possible even indexes
%  on ${{X}^{l}},\,\,l<k-1$ of ${{X}^{[k]}}$ and on ${{X}^{k-2}}$ of subtrees ${{v}_{11}}{{X}^{[k-1]}}$ and %${{v}_{12}}{{X}^{[k-1]}}$.
%Formulation 2.
An element $g=(g_1, g_2)\sigma^{i}$ of $G_k$, $i\in\{0,1\}$ belongs to $G'_k$ iff $g$ has even index on ${{X}^{l}}$ for all $l<k-1$ and elements $g_1, g_2$ have even indexes on ${{X}^{k-1}}$, that is equally matched to $g_1, g_2 \in G_{k-1}$.
 \end{lemma}
 \begin{proof}
The proof implies from Lemma \ref{comm} and Lemma \ref{L_k_comm_criteria}.
%\note{add proof}
 \end{proof}

Using this structural property of $(Syl_2 A_{2^k})'$ we deduce a following result.
\begin{theorem}\label{_comm_G_k_eq_[G_k,G_k]} Commutator subgroup $G'_k$ coincides with set of all commutators, put it differently
$\syl[k]'=\{[f_1,f_2]\mid f_1\in \syl[k], f_2\in \syl[k]\}$.
%Commutator width of $G_k$ is equal to 1.
%The set $G'_k$ coincides with set of all commutators of $G_k$.
\end{theorem}
\begin{proof}
For the case $k=1$ we have $G_1' = \langle e \rangle$. So, further we consider case $k\geq 2$. In order to prove this Theorem we fix arbitrary element $w\in G_k'$  and then we represent this element as commutator of elements from $G_k$.

We already know by Lemma~\ref{L_k_comm_criteria} that every element $w\in\syl[k]'$ we can represent as follow
\begin{align*}
w=(r_1, r_1\m x),
\end{align*}
where  $r_1  \in G_{k-1}$ and $ x \in B'_{k-1}$. %, denote a product $r_1 r_2$ as $ x$, then $r_2=r_1^{-1} x $ which is in form mentioned above, for some $r_1\in \syl[k-1]$.
By proposition~\ref{_comm_F_k_eq_[L_k,F_k]} %(OR \ref{_comm_L_k_eq_[L_k,L_k]} ? )
%about form of commutator $B'_k$
we have $ x = [f,g]$ for some $f\in\aut[k-1]$ and $g\in\syl[k-1]$. Therefore
\begin{align*}
w=(r_1, r_1\m [f,g]).
\end{align*}

By the Corollary~\ref{c_2_wr_b_elem_repr} we can represent $w$ as commutator of
\begin{align*}
(e,a_{1,2})\sigma \in \aut[k] \mbox{ and } (a_{2,1}, a_{2,2}) \in \aut[k],
\end{align*}
where
\begin{align*}
a_{2,1} &= (f\m)^{r_1\m},\\
a_{2,2} &= r_{1} a_{2,1},\\
a_{1,2} &= g^{a_{2,2}\m}.
\end{align*}
%We note that as $g \in G_{k-1}$ then by proposition~\ref{B_k_criteria} we have $(e,a_{1,2})\sigma \in \syl[k]$.
%So, we have $w = [(e,a_{1,2})\sigma, (a_{2,1}, a_{2,2})]$.
It is only left to show that $(e,a_{1,2})\sigma, (a_{2,1}, a_{2,2}) \in G_k$.

%In order to use Proposition~\ref{B_k_criteria} we need to show the following
In order to use Proposition~\ref{B_k_criteria} we note that
\begin{align*}
%a_{2,2} a_{2,1} &=& r_1 a_{2,1}^2\in \syl[k-1] \\
a_{1,2} = g^{a_{2,2}\m} &\in G_{k-1}\mbox{ by Proposition~\ref{G_k_is_normal_in_B_k}}.\\
a_{2,1} a_{2,2} = a_{2,1} r_1 a_{2,1} = r_1 [r_1, a_{2,1}] a_{2,1}^2 &\in \syl[k-1]\mbox{ by Proposition~\ref{comm_F_k_is_subgroup_of_L_k} and Proposition~\ref{B'_k and B^2_k}}.
\end{align*} %\ref{_comm_L_k_eq_[L_k,L_k]} %  (NOT \ref{_comm_F_k_eq_[L_k,F_k]})
% because of Proposition~\ref{comm_F_k_is_subgroup_of_L_k}
%%($B_{k-1}' < G_{k-1}$)
%and Proposition~\ref{B'_k and B^2_k}
%($B_{k-1}^2 < B_{k-1}'$)
%.

%so conditions of Proposition \ref{comm_F_k_is_subgroup_of_L_k} holds.
% Note that using Proposition~\ref{comm_F_k_is_subgroup_of_L_k} about subgroup of $B'_{k-1}$ into $G_{k-1}$ (in more details $B'_{k-1} \lhd G_{k-1}$) we have inclusion of element $[r_1,a_{2,1}]\in G_{k-1}$, where $r_1\in G_{k-1}$ by condition mentioned above. And $a_{2,1}^2 \in G_{k-1}$ according to Proposition \ref{B'_k and B^2_k}.
%Hence, the product $r_1 [r_1, a_{2,1}] a_{2,1}^2\in G_{k-1}$ too.
%\begin{align*}
%a_{2,1} &=& (f\m)^{(r_1)\m}\\
%a_{2,2} &=& r_{1} a_{2,1}\\
%a_{1,2} &=& g^{a_{2,2}\m}\in G_{k-1}\mbox{ by proposition~\ref{G_k_is_normal_in_B_k}  } G_{k-1}\lhd B_{k-1} %\mbox{ so } g^{a_{2,2}\m}\in G_{k-1}\\
%\mbox{Therefore } a_{1,2}^{a_{2,2}} & \in &  G_{k-1}\mbox{ so conditions of proposition~ %\ref{_comm_F_k_eq_[L_k,F_k]}  holds } \\
%a_{2,2} a_{2,1} &=& r_1 a_{2,1}^2\in \syl[k-1]
%\end{align*}
So we have $(e,a_{1,2})\sigma \in \syl[k]$ and $(a_{2,1}, a_{2,2}) \in \syl[k]$. % it means solution of form $w=(r_1, r_1\m x)$ are confirmed the condition of this Proposition.
\end{proof} % Lemma \ref{comm} (better by

\begin{corollary}
Commutator width of the group $Syl_2 A_{2^k}$ equal to $1$ for $k\geq 2$.
%Note that it follow from Theorem~\ref{_comm_G_k_eq_[G_k,G_k]} that $cw(Syl_2 A_{2^k})=1$
\end{corollary}

%\begin{corollary}\label{_comm_B_k_eq_[b_k,b_k]}
%$\syl[k]'=\{[f_1,f_2]\mid f_1\in \syl[k], f_2\in \syl[k]\}$.
%Commutator width of $B_k$ is equal to 1.
%The set $B'_k$ coincides with set of all commutators of $B_k$.
%\end{corollary}
%\begin{proof}
%As it was proved in Theorem \ref{max} $G_k \simeq  B_{k-1} \ltimes W_{k-1} $ so $B_{k-1}$ $B_{k-2}$ are subgroups %of $G_k$. According  Theorem 1.2. from \cite{nikolov} we know that $cw G \geq max{cw(A), \frac{1}{n} cw(B)}$ %where $A$ is active, thus $cw( B_{k-2})=1$.
%\end{proof}

\begin{theorem}
 The centralizer of  $Syl_{2}{{S}_{{{2}^{{{k}_{i}}}}}}$ \cite{Kal, Sk} with ${{k}_{i}}>2$, in $Sy{{l}_{2}}{{S}_{n}}$  is isomorphic to ${}^{Syl_2 {A}_{n}}/{}_{Sy{{l}_{2}}{{S}_{{{2}^{{{k}_{i}}}}}}}  \times Z( Syl_2 {S}_{2^{{{k}_{i}}}}) $.
 \end{theorem}
\begin{proof} Actually, the action of active group $A \simeq \wr^{k_i-1}_{i=1} C_2$ of $Syl_2 {S}_{2^{{{k}_{i}}}}$ on $X^{k-1}$ is transitive since the orbit of $A$ on $X^{k-1}$ is one. Then $Z ( \wr^{k_i}_{j=1} C_2) \simeq C_2$ results by formula from Corollary 4.4 \cite{Meld}.

For any no trivial automorphism $\alpha $ from $Aut{{X}^{[{{k}_{i}}]}} \simeq Sy{{l}_{2}}{{S}_{{{2}^{{{k}_{i}}}}}}$ there exists a vertex ${{v}_{jm}}$, where v.p. from $\alpha$ is active. Thus, there exists  graph path $r$ connecting  the root  ${{v}_{0}}$  with a vertex of ${{X}^{{{k}_{i}}}}$ and pathing through  the ${{v}_{jm}}$.  We can choose vertex ${{v}_{li}}$ on $r$ such that $l\ne j$. So there exists $\beta \in Aut{{X}^{{[{k}_{i}]}}}$ that has active v.p. in ${{v}_{li}}$. Then we have $\alpha \beta \ne \beta \alpha $. The center of $Aut{{X}^{[{{k}_{i}}]}}$  is isomorphic to $C_2$. As a result we have ${{C}_{Aut{{S}_{n}}}}(Sy{{l}_{2}}{{S}_{{{2}^{{{k}_{i}}}}}})\simeq {}^{Sy{{l}_{2}}{{S}_{n}}}/{}_{Sy{{l}_{2}}{{S}_{{{2}^{{{k}_{i}}}}}}} \times Z( Syl_2 {S}_{2^{{{k}_{i}}}})$.
\end{proof}	

Let us present new operation $\boxtimes $ (similar to that is in \cite{Dm}) as an even subdirect product of $Syl{{S}_{{{2}^{i}}}}$, $n = {{2}^{{{k}_{0}}}}+{{2}^{{{k}_{1}}}}+...+{{2}^{{{k}_{m}}}}$, $0\le {{k}_{0}}<{{k}_{1}}<...<{{k}_{m}}$.

\begin{theorem}
 The centralizer of $Syl_{2}{{A}_{{{2}^{{{k}_{i}}}}}}$ with ${{k}_{i}}>2$, in $Sy{{l}_{2}}{{A}_{n}}$ is isomorphic to $Sy{{l}_{2}}{{S}_{{{2}^{{{k}_{0}}}}}}\boxtimes ... \boxtimes  Sy{{l}_{2} S_{{{2}^{{{k}_{i-1}}}}}}\boxtimes Sy{{l}_{2} S_{{{2}^{{{k}_{i+1}}}}}}\boxtimes  ...\boxtimes Syl_2 S_{{{2}^{{{k}_{m}}}}}  \boxtimes Z( Syl_2 S_{2^{{k_{i}}}})$.
 \end{theorem}

\begin{proof}  We consider $G_k$ as a normal subgroup of $\wr^{k}_{j=1} C_2$.
Actually, the action of subgroup $A=B_{k_i-1}$ of $G_{k_i} \simeq Syl_{2}{{A}_{{{2}^{{{k}_{i}}}}}}$ on $X^{k-1}$ is transitive since the orbit of $A$ on $X^{k-1}$ is one.

%%Then $Z (Syl_{2}{{A}_{{{2}^{{{k}_{i}}}}}})  \simeq C_2$ results by formula from Corollary 4.4 \cite{Meld}.
%%\rightthreetimes

There exists a vertex ${{v}_{jm}}$ for any no trivial automorphism $\alpha $ from $G_{{k}_{i}} \simeq Sy{{l}_{2}}{{A}_{{{2}^{{{k}_{i}}}}}}$, where v.p. from $\alpha$ is active. Thus there exists  graph path $r$ connecting  the root  ${{v}_{0}}$  with a vertex of ${{X}^{{{k}_{i}}}}$ and pathing through  the ${{v}_{jm}}$.  We can choose vertex ${{v}_{li}}$ on $r$ such that $l\ne j$. So there exists $\beta \in Aut{{X}^{{[{k}_{i}]}}}$ that has active v.p. in ${{v}_{li}}$. Then we have $\alpha \beta \ne \beta \alpha $. Since $Syl_2 A_{2^k} \simeq Aut X^{[k]}$, consequently the center of $Syl_2 A_{2^{k_i}}$  is isomorphic to $C_2$. As a result we have ${{C}_{Aut{{A}_{n}}}}(Sy{{l}_{2}}{{S}_{{{2}^{{{k}_{i}}}}}})\simeq
Sy{{l}_{2}}{{S}_{{{2}^{{{k}_{0}}}}}}\boxtimes ... \boxtimes  Sy{{l}_{2} S_{{{2}^{{{k}_{i-1}}}}}}\boxtimes Sy{{l}_{2} S_{{{2}^{{{k}_{i+1}}}}}}\boxtimes  ...\boxtimes Syl_2 S_{{{2}^{{{k}_{m}}}}}  \boxtimes Z(  A_{2^{{k_{i}}}})$.
\end{proof}	
Also we note that derived length of $Syl_2 A_2^k$ is not always equal to $k$ as it was said in Lemma 3 of \cite{Dm} because in case $A_{2^k}$ if $k=2$ its $Syl_2 A_4 \simeq K_4$ but $K_4$ is abelian group so its derived length is 1.

\section{Conclusion }
  The commutator width of Sylow 2-subgroups of alternating group ${A_{{2^{k}}}}$, permutation group ${S_{{2^{k}}}}$ and Sylow $p$-subgroups of $Syl_2 A_p^k$ ($Syl_2 S_p^k$) is equal to 1. Commutator width of permutational wreath product $B \wr C_n$, were $B$ is arbitrary group, was researched.
%  \begin{Conclusion}
%  \end{Conclusion}

\end{document}